\documentclass[11pt,oneside]{amsart}

\usepackage[left=1in,right=1in,top=1in,bottom=1in,marginparwidth=0.8in]{geometry}

\usepackage{microtype}  

\usepackage{mathtools}
\usepackage{bbm}  

\swapnumbers

\usepackage{hyperref}

\usepackage{amssymb,amsfonts,amsmath,amsthm}
\usepackage[mathscr]{eucal}
\usepackage[alphabetic]{amsrefs}
\usepackage{stmaryrd}
\usepackage{colonequals}

\usepackage[ps,matrix,arrow,curve,cmtip]{xy}

\title{Free colimit completion in $\infty$-categories}
\author{Charles Rezk}
\date{ \today}
\address{Department of Mathematics \\
University of Illinois Urbana-Champaign \\ 
Urbana, IL}
\email{rezk@illinois.edu}

\numberwithin{equation}{section}

\makeatletter
  \let\c@subsection\c@equation
\makeatother

\theoremstyle{plain}   

\newtheorem{thm}[subsection]{Theorem}
\newtheorem*{thm*}{Theorem}
\newtheorem{prop}[subsection]{Proposition}
\newtheorem*{prop*}{Proposition}
\newtheorem{cor}[subsection]{Corollary}
\newtheorem*{cor*}{Corollary}
\newtheorem{lemma}[subsection]{Lemma}
\newtheorem*{lemma*}{Lemma}

\newtheorem*{claim*}{Claim}

\theoremstyle{remark}
\newtheorem{rem}[subsection]{Remark}    
\newtheorem*{rem*}{Remark}
\newtheorem{exam}[subsection]{Example}
\newtheorem*{exam*}{Example}

\theoremstyle{plain}

\begin{document}


\newcommand{\margnote}[1]{\mbox{}\marginpar{\tiny\hspace{0pt}#1}}

\def\lambada{\lambda}


\newcommand{\id}{\operatorname{id}}
\newcommand{\colim}{\operatorname{colim}}
\newcommand{\llim}{\operatorname{lim}}
\newcommand{\laxlim}{\operatorname{laxlim}}
\newcommand{\oplaxlim}{\operatorname{oplaxlim}}
\newcommand{\Cok}{\operatorname{Cok}}
\newcommand{\Ker}{\operatorname{Ker}}
\newcommand{\Image}{\operatorname{Im}}
\newcommand{\op}{{\operatorname{op}}}
\newcommand{\Aut}{{\operatorname{Aut}}}
\newcommand{\End}{{\operatorname{End}}}
\newcommand{\Hom}{{\operatorname{Hom}}}

\newcommand*{\ra}{\rightarrow}
\newcommand*{\lra}{\longrightarrow}
\newcommand*{\xra}{\xrightarrow}
\newcommand*{\la}{\leftarrow}
\newcommand*{\lla}{\longleftarrow}
\newcommand*{\xla}{\xleftarrow}

\newcommand{\ho}{\operatorname{ho}}
\newcommand{\hocolim}{\operatorname{hocolim}}
\newcommand{\holim}{\operatorname{holim}}

\newcommand*{\realiz}[1]{\left\lvert#1\right\rvert}
\newcommand*{\len}[1]{\left\lvert#1\right\rvert}
\newcommand{\set}[2]{{\{\,#1\,\mid\,#2\,\}}}
\newcommand{\bigset}[2]{\left\{\,#1\;\middle|\;#2\,\right\}}
\newcommand*{\tensor}[1]{\underset{#1}{\otimes}}
\newcommand*{\pullback}[1]{\underset{#1}{\times}}
\newcommand*{\powser}[1]{[\![#1]\!]}
\newcommand*{\laurser}[1]{(\!(#1)\!)}
\newcommand{\ndiv}{\not|}
\newcommand{\pairing}[2]{\langle#1,#2\rangle}

\newcommand{\lrtensor}[3]{{\mathstrut}^{#1}\!\otimes_{#2}\!{\mathstrut}^{#3}}
\newcommand{\ltensor}[2]{{\mathstrut}^{#1}\!\otimes_{#2}}
\newcommand{\rtensor}[2]{\otimes_{#1}\!{\mathstrut}^{#2}}

\newcommand{\F}{\mathbb{F}}
\newcommand{\Z}{\mathbb{Z}}
\newcommand{\N}{\mathbb{N}}
\newcommand{\R}{\mathbb{R}}
\newcommand{\Q}{\mathbb{Q}}
\newcommand{\C}{\mathbb{C}}

\newcommand{\point}{{\operatorname{pt}}}
\newcommand{\Map}{\operatorname{Map}}
\newcommand{\eev}{\wedge}
\newcommand{\sm}{\wedge} 

\newcommand*{\mc}{\mathcal}
\newcommand*{\msc}{\mathscr}
\newcommand*{\mf}{\mathfrak}
\newcommand*{\mr}{\mathrm}
\newcommand*{\mb}{\mathbb}
\newcommand*{\ul}{\underline}
\newcommand*{\ol}{\overline}
\newcommand*{\wt}{\widetilde}
\newcommand*{\wh}{\widehat}
\newcommand*{\mtt}{\mathtt}
\newcommand*{\ms}{\mathsf}
\newcommand*{\mbf}{\mathbf}

\newcommand{\dfn}{\textbf}

\def\noloc{\;{:}\,}

\newcommand{\defeq}{\colonequals}

\newcommand{\forcepar}{\mbox{}\par}

\newcommand{\icat}{\mr{Cat}_\infty}
\newcommand{\igpd}{\msc{S}}

\newcommand{\Psh}{\operatorname{PSh}}
\def\PSh{\Psh}
\newcommand{\Fun}{\operatorname{Fun}}
\newcommand{\Ind}{\operatorname{Ind}}

\newcommand{\Idem}{\mr{Idem}}

\newcommand{\Set}{\mr{Set}}

\newcommand{\Filt}{\operatorname{Filt}}
\newcommand{\coFilt}{\operatorname{coFilt}}
\newcommand{\Sm}{\mr{Sm}}

\newcommand{\LPath}{\operatorname{LPath}}
\newcommand{\Path}{\operatorname{Path}}

\begin{abstract}
We show how several useful properties of $\mathrm{Ind}$-constructions in
$\infty$-categories extend to arbitrary free colimit completion
constructions.  In doing so we discover the useful notion of a
\emph{regular class} of $\infty$-categories.
\end{abstract}

\maketitle


\section{Introduction}

It is well-known that the $\infty$-category $\Psh(C)$ of presheaves of
$\infty$-groupoids on $C$ is the ``free colimit completion''
of $C$.  
  More generally, there is a ``free $\mc{F}$-colimit
completion'' $\Psh^{\mc{F}}(C)$ for a given class $\mc{F}$ of
$\infty$-categories, which 
can be exhibited as the full subcategory of $\PSh(C)$ generated by
representable presheaves under 
$\mc{F}$-colimits, as described
by Lurie in \cite{lurie-higher-topos}*{5.3.6}.  Note that ``free''
here means we are not merely adjoining some colimits, but rather that the
construction is characterized by a universal property.

In special cases we have more.  For instance, when $\mc{F}$ is the
class of \emph{$\kappa$-filtered} $\infty$-categories for some regular
cardinal $\kappa$, then the free $\mc{F}$-colimit completion
admits a rather more explicit description: it is  $\Ind_\kappa(C)$,
the full subcategory of presheaves $X$ on $C$ which represent a right
fibration $C/X\ra C$ such that $C/X$ is $\kappa$-filtered.  Here $C/X$
is the ``point category'' (or ``category of elements'') of the functor
$X\colon C^\op\ra \igpd$. 

Furthermore, there is a very useful ``recognition principle'' for such
categories: any $\infty$-category $A$ which is generated under
$\kappa$-filtered colimits by a full subcategory $C\subseteq A$ of
``$\kappa$-compact objects'' is canonically equivalent to
$\Ind_\kappa(C)$ \cite{lurie-higher-topos}*{5.3.5}.  These
$\Ind$-categories are the basis of the theory of accessible 
$\infty$-categories.

This note addresses the question: to what extent can these pleasant
properties of $\Ind_\kappa(C)$ be extended to arbitrary free colimit
completions?  The answer is: in some sense, pretty much all of them.
Our results are encapsulated by the following observation.
\begin{quote}
  \emph{Whether a presheaf $X$ is in $\Psh^{\mc{F}}(C)$ depends only on its
   point category $C/X$.}
\end{quote}
This in turn leads to the useful notion of the \emph{regular closure}
of $\mc{F}$, which is the collection of all categories which can
appear as $C/X$ for some $C$ and some $X\in \PSh^{\mc{F}}(C)$.  

Here is a brief summary of our results.
\begin{itemize}
\item  Any class $\mc{F}$ of small $\infty$-categories can be enlarged
  to a  \emph{regular closure} $\ol{\mc{F}}$, which consists of
 $C$ whose free $\mc{F}$-colimit completion  $\Psh^{\mc{F}}(C)$
 contains a terminal presheaf 
  (\S\ref{sec:regular-closure}).    We say that 
  $\mc{F}$ is a \emph{regular class} when $\mc{F}=\ol{\mc{F}}$
  (\S\ref{sec:regular-class}). 
\item We get an explicit criterion for describing the free
  $\mc{F}$-colimit completion $\Psh^{\mc{F}}(C)\subseteq \Psh(C)$ as a
  full subcategory, much as for $\Ind_\kappa(C)$: a presheaf $X$ is in
  $\Psh^{\mc{F}}(C)$ if and only if $C/X$ is in the regular closure
  of $\mc{F}$, where $C/X\ra C$ is the right fibration classified by
  $X$ \eqref{prop:regular-criterion-for-colim-closure}. 
\item Thus free $\mc{F}$-colimit completion  depends only on the
  regular closure $\ol{\mc{F}}$, and in fact the regular
  classes precisely correspond to  possible ``types'' of free colimit
  completion
  \eqref{prop:reg-closures-corr-to-free-colimit-completions}. 
\item We find that any $\infty$-category which has all
  $\mc{F}$-colimits also has all $\ol{\mc{F}}$-colimits, and any functor
  between such which preserves all $\mc{F}$-colimits also preserves all
  $\ol{\mc{F}}$-colimits \eqref{prop:have-and-preserve-colims-regular}.
\item Furthermore, 
  the previous is in   some sense the best possible: $\ol{\mc{F}}$ is
  the largest class for which we can have such a collection of results
  \eqref{prop:have-and-preserve-colims-regular}.
\item The $\infty$-groupoids in a regular class $\ol{\mc{F}}$, as well as the
  underlying weak homotopy types of objects of $\ol{\mc{F}}$, are precisely
  those in the full subcategory of $\infty$-groupoids generated by the
  terminal object under $\mc{F}$-colimits
  \eqref{prop:regular-class-infty-gpd}. 

\item There is a recognition principle for free $\mc{F}$-colimit completion
   generalizing that for $\Ind$-categories, which is
  stated in terms of the evident notion of \emph{$\mc{F}$-compact
    object}.
  \eqref{prop:recognition-free-colimit-completion}.  
\item It is straightforward to produce examples of regular classes,
  which include the familiar classes of \emph{$\kappa$-filtered} and
  \emph{sifted} $\infty$-categories, but also many others which have
  not been much studied (\S\ref{sec:cut-out-reg-classes},
  \S\ref{sec:examples}). 
\end{itemize}
Given all this, it would be very desirable to have methods for calculating
regular closures of various classes of interest.    Further study is
needed!

I came to this while working on a project to understand
generalizations of $\Ind$-constructions and accessible
$\infty$-categories, motivated by the 1-categorical work of
\cite{adamek-borceux-lack-rosicky-classification-acc-cat}.  The idea
is to look at classes $\mc{F}$ of $\infty$-categories  characterized
by how $\mc{F}$-colimits of $\infty$-groupoids preserve a fixed
collection of types of limit \cite{rezk-gen-accessible-infty-cats},
e.g., much as $\kappa$-filtered colimits preserve $\kappa$-small limits of
$\infty$-groupoids, or sifted colimits preserve finite products of
$\infty$-groupoids.   Such classes, called \emph{filtering classes},
are defined and briefly discussed here in \eqref{subsec:filtration-and-cofiltration}
In the course of this I realized that for many
purposes there is nothing special about classes described in terms of
such limit preservation.   It is fair to say that none of the results
here are particularly deep, and that some of these results have been
surely noticed by others.  However, the picture they make is pleasant and
perhaps surprising, so it seems worthwhile to lay out the story in detail.
I note that there is every reason to expect that
most of  the $\infty$-categorical results described here  have
1-categorical analogues.  However, I have not attempted to trace this
out explicitly.

Here is a brief outline of the paper.  After reviewing basic facts about presheaf
categories and their role as free colimit completions
(\S\ref{sec:basic-notions} and \S\ref{sec:free-colimit-completion}),
we introduce the notions of regular closure
\S\ref{sec:regular-closure} and regular class \S\ref{sec:regular-class}
and show how they precisely
answers certain questions about existence and/or preservation of
colimits.  In \S\ref{sec:explicit-description-free-colimit} we use
these to give an explicit description of free 
colimit completions, and in particular note that forming free
completions is compatible with forming
slices \eqref{cor:colim-closure-slice}.  We review the relation
between regular closure and cofinal functors in
\S\ref{sec:cofinality}, the relation between regular classes and
$\infty$-groupoids in \S\ref{sec:reg-class-infty-gpd}, and a
general way to find regular classes in
\S\ref{sec:cut-out-reg-classes}.  The recognition principal for free
colimit completion (generalizing the one for $\Ind$-categories) is
proved in \S\ref{sec:recognition-principle}.  We give a number of
(mostly standard) examples of these phenomena in \S\ref{sec:examples}.
The final section \S\ref{sec:functors-preserving-colimits} is an
appendix, proving a criterion which reduces preservation of colimits
by functors to stability of full subcategories under colimits, whose
proof was described to me by Maxime Ramzi.

Thanks to Maxime Ramzi and Sil Linksens for help with
(\S\ref{sec:functors-preserving-colimits}), and others on the
algebraic topology Discord server who answered my queries, including
Tim Campion and Dylan Wilson.

\section{Basic $\infty$-categorical notions}
\label{sec:basic-notions}

\subsection{Universes}

We work with respect to a chosen \dfn{universe} of small simplicial
sets, which determines an $\infty$-category $\icat$ of \dfn{small
$\infty$-categories}, together with a full subcategory $\igpd\subseteq
\icat$ of \dfn{small $\infty$-groupoids}.  We say that an
$\infty$-category is 
\dfn{locally small} if its mapping spaces are equivalent to small
$\infty$-groupoids.

We also discuss $\infty$-categories which are
not small.  These may be imagined to live in some higher universe, but
I will not need to refer explicitly to a hierarchy of 
universes as in \cite{lurie-higher-topos}*{1.2.15}.  However, some of
the results I use do rely on universe-hopping, most
notably Lurie's construction of free colimit completions
\eqref{thm:psh-f-free-colimit}, and his related embedding 
theorem \eqref{thm:colimit-embedding}. 

\subsection{Colimits}
\label{subsec:colimits}

By a \dfn{small colimit}, I mean a colimit of a functor $J\ra A$ where
$J$ is a small $\infty$-category.  
I say that an $\infty$-category $A$ is \dfn{cocomplete} if it has all
\emph{small} colimits, and \dfn{complete} if it has all \emph{small} limits.

More generally given a class $\mc{F}$ of $\infty$-categories I will
speak of \dfn{$\mc{F}$-colimits}, i.e., colimits of functors $J\ra A$
where $J\in \mc{F}$.  Thus, I can speak of an $\infty$-category $A$ which
\dfn{has $\mc{F}$-colimits}, i.e., is such that every $J\ra A$ with
$J\in \mc{F}$ admits a colimit.

Given a fully faithful functor $f\colon A\ra B$ (e.g., the inclusion
of a full subcategory), I say that $f$ is \dfn{stable under
  $\mc{F}$-colimits} if (i)  $A$ has $\mc{F}$-colimits
and (ii) $f$ preserves all $\mc{F}$-colimits in $A$.  If $A\subseteq B$ is an
inclusion of a full subcategory, I'll just say that $A$ is stable in $B$
under $\mc{F}$-colimits.

Finally, given a full subcategory $A\subseteq B$ of an
$\infty$-category $B$ which has $\mc{F}$-colimits, the subcategory
\dfn{generated by $A$ under $\mc{F}$-colimits} is the smallest full
subcategory $A'\subseteq B$ containing $A$ which is stable under
$\mc{F}$-colimits.

\subsection{Presheaves}
\label{subsec:presheaves}

Given a small $\infty$-category $C$, we write $\Psh(C)\defeq
\Fun(C^\op,\igpd)$ for the category of \dfn{presheaves} of
$\infty$-groupoids on $C$ (rather than Lurie's notation
$\msc{P}(C)$ of \cite{lurie-higher-topos}*{5.1}).  I denote the Yoneda
functor by $\rho_C\colon 
C\ra \Psh(C)$, or just $\rho$ if the context is clear.  Recall that
$\rho$ is fully faithful and that $\Psh(C)$ is complete and
cocomplete.

\subsection{Slices of presheaves}
\label{subsec:slices-of-presheaves}

Given a presheaf $X\in \Psh(C)$, I write
\[
C/X \defeq C\times_{\Psh(C)} \Psh(C)_{/X}
\]
for the evident pullback of the slice projection $\Psh(C)_{/X}\ra
\Psh(C)$ along $\rho$, and $\pi_X\colon C/X\ra C$ for the evident
projection.  The composite $\rho_C\pi_X\colon C/X\ra \Psh(C)$ comes
with an extension to a colimit functor $\wt\rho\colon (C/X)^\rhd\ra
\Psh(C)_{/X}$, which 
exhibits $X$ \emph{tautologically} as a colimit of $\rho\pi_X$
\cite{lurie-higher-topos}*{5.1.5.3}.

In particular, the colimit of $\rho\colon C\ra \Psh(C)$ is a terminal
presheaf, since $C/1\approx C$.

\begin{rem} $C/X$ may be regarded as
an ``$\infty$-category of elements'' or \dfn{point category} of $X$, by
analogy with the 
1-categorical analogue.  The projection $\pi_X\colon C/X\ra C$ is a
right fibration, representing the unstraightening of the functor
$X\colon C^\op\ra \igpd$.
\end{rem}

The evident functor $C/X\ra \Psh(C)_{/X}$ induces by
restriction an equivalence
\[
\kappa\colon \Psh(C)_{/X}\xra{\sim} \Psh(C/X),
\]
i.e., every slice of a presheaf category is a presheaf category on a
category of elements \cite{lurie-higher-topos}*{5.1.6.12}.  Under this
equivalence, the
forgetful functor $\Psh(C)_{/X}\ra \Psh(C)$ corresponds to a functor
denoted  $\wh{\pi_X}\colon \Psh(C/X)\ra \Psh(C)$, which is necessarily
colimit preserving and which comes with a natural isomorphism
$\wh{\pi_X} \circ \rho_{C/X}\approx \rho_C$ of functors $C/X\ra \Psh(C)$.

Finally, note that if $X\in \Psh(C)$ and $\wt{Y}\defeq (f\colon Y\ra X)\in
\Psh(C)_{/X}$, then we have an equivalence
\[
(\Psh(C)_{/X})_{/\wt{Y}} \approx \Psh(C)_{/Y},
\]
which when combined with the equivalence $\kappa\colon
\Psh(C)_{/X}\approx \Psh(C)$ restricts to an equivalence of full subcategories
\[
(C/X)/\wt{Y} \approx C/Y.
\]

\section{Free colimit completion}
\label{sec:free-colimit-completion}

The Yoneda functor $\rho\colon C\ra \Psh(C)$ exhibits the \dfn{free
  colimit completion} of $C$.
\begin{thm}\label{thm:psh-free-colimit}\cite{lurie-higher-topos}*{5.1.5.6}
  For any cocomplete $\infty$-category $A$, restriction along $\rho$
  induces an equivalence
  \[
  \Fun(\Psh(C),A)\supseteq \Fun^{\mr{colim}}(\PSh(C),A) \ra \Fun(C,A)
\]
from the category of colimit preserving functors $\PSh(C)\ra A$ to the
category of functors $C\ra A$.

In particular, any functor $f\colon
C\ra A$ admits an essentially unique extension $\wh{f}\colon
\PSh(C)\ra A$ to a colimit preserving functor equipped with a natural
isomorphism $\wh{f}\rho\approx f$.
\end{thm}

As a consequence, $\Psh(C)$ contains the \dfn{universal $C$-colimit},
which is just the terminal presheaf.
\begin{cor}\label{cor:universal-colim}
  Let $f\colon C\ra A$ be any functor from a small $\infty$-category
  to a cocomplete $\infty$-category.  Then the colimit of $f$ in $A$
  is equivalent to $\wh{f}(1)$, where $\wh{f}\colon \Psh(C)\ra A$ is
  any colimit preserving extension of $f$ along $\rho$.
\end{cor}
\begin{proof}
  Since $\wh{f}\colon \Psh(C)\ra A$ preserves colimits, and $1\approx
  \colim_C \rho$.
\end{proof}

Let $\mc{F}\subseteq \icat$ be a class of small $\infty$-categories.
Given a small $\infty$-category $C$, let 
$\Psh^{\mc{F}}(C)\subseteq \Psh(C)$ denote the full subcategory
generated by representable presheaves under
  $\mc{F}$-colimits. That is, $\Psh^{\mc{F}}(C)$ is the smallest full
subcategory of presheaves which (i) contains the image of the Yoneda
functor $\rho\colon C\ra \Psh(C)$ and (ii) is stable under
$\mc{F}$-colimits.  The restriction $\rho\colon C\ra \Psh^{\mc{F}}(C)$
exhibits the \dfn{free $\mc{F}$-colimit completion} of $C$.

\begin{thm}\cite{lurie-higher-topos}*{5.3.6.2}\label{thm:psh-f-free-colimit}
If 
  $\mc{F}\subseteq \icat$ is a class of small $\infty$-categories, and
  if $A$ is an
  $\infty$-category which has $\mc{F}$-colimits, then restriction
  along $\rho$ exhibits  an equivalence
  \[
  \Fun(\Psh^{\mc{F}}(C),A)\supseteq
  \Fun^{\mc{F}\mr{-colim}}(\Psh^{\mc{F}}(C), A)\ra \Fun(C, A) 
\]
from the category of $\mc{F}$-colimit preserving functors
$\Psh^{\mc{F}}(C)\ra A$ to $\Fun(C,A)$.

In particular, any functor $f\colon C\ra A$ admits an essentially
unique extension $\wh{f}\colon \Psh^{\mc{F}}(C)\ra A$ to an
$\mc{F}$-colimit preserving functor equipped with a natural
isomorphism $\wh{f}\rho\approx f$.
\end{thm}

We refer to Lurie for the proof, but note that his proof both provides and
relies on the following,  which we will use later. 

\begin{thm}[Embedding theorem]\cite{lurie-higher-topos}*{5.3.6.2}\label{thm:colimit-embedding}
 Given any classes $\mc{F}\subseteq \mc{G}$ of $\infty$-categories and an
 $\infty$-category $A$ which has $\mc{F}$-colimits, there exists a
 fully faithful functor $i\colon A\rightarrowtail B$ such that (i) $B$
 has all $\mc{G}$-colimits and (ii) $i$ preserves all $\mc{F}$-colimits.
\end{thm}
\begin{proof}[Sketch proof]
  Construct $B$ as a full subcategory (in fact, a localization)
  of $\Fun(A^\op,\wh{\igpd})$ 
  where $\wh\igpd$ is an $\infty$-category of $\infty$-groupoids in a
  suitably large universe.
\end{proof}

If the class $\mc{F}$ of small $\infty$-categories is essentially small relative to our chosen
universe (i.e., there is a set $\mc{F}'$ such that every object of
$\mc{F}$ is equivalent to one in $\mc{F}'$), then so is any
$\mc{F}$-colimit completion of a small category.
\begin{prop}\label{prop:small-colimit-completion-is-small}
  Let $C$ be a small $\infty$-category and let $\mc{F}\subseteq \icat$
  be an essentially small class of $\infty$-categories.  Then the
  free $\mc{F}$-colimit completion $\Psh^{\mc{F}}(C)$ is also essentially small.
\end{prop}
\begin{proof}
For each ordinal $\lambda$ define subcategories
$\mc{P}^\lambda\subseteq \PSh(C)$, with
\begin{itemize}
\item $\mc{P}^0=$ the essential image of the Yoneda functor
  $\rho\colon C\ra \PSh(C)$,
  \item $\mc{P}^\lambda=\bigcup_{\mu<\lambda} \mc{P}^\mu$ for any
    limit ordinal $\lambda$, and
  \item $\mc{P}^{\lambda+}=$ the full subcategory spanned by all
    colimits in $\PSh(C)$ of functors $J\ra \mc{P}^\lambda$, where
    $J\in \mc{F}$.   
  \end{itemize}
  Note that $\Psh^{\mc{F}}(C)= \bigcup_\lambda \mc{P}^\lambda$, and
  that since $\mc{F}$ is essentially small, so is every $\mc{P}^\lambda$.
  Choose a regular cardinal $\kappa$ such that every $J\in \mc{F}$ is
  equivalent to some simplicial set with fewer than $\kappa$
  non-degenerate cells.  Then $\mc{P}^\kappa\subseteq \Psh(C)$ is
  stable under $\mc{F}$-colimits, and so is equal to
  $\PSh^{\mc{F}}(C)$, since any functor $J\ra \mc{P}^\kappa$ such that
  $J$ has fewer than $\kappa$ non-degenerate cells factors through
  some $\mc{P}^\lambda$ with $\lambda<\kappa$.
\end{proof}

\section{Regular closure}
\label{sec:regular-closure}

Let $\mc{F}\subseteq\icat$ be a class of small $\infty$-categories.
The \dfn{regular closure} of $\mc{F}$ is defined to be the class
\[
  \ol{\mc{F}} \defeq \set{C\in \icat}{\text{$\Psh^{\mc{F}}(C)$
      contains the terminal presheaf}}.  
\]
Note that if $C\in \mc{F}$, then necessarily $\Psh^{\mc{F}}(C)$
contains the terminal presheaf, since the terminal presheaf is the
colimit of $\rho_C$ \eqref{subsec:slices-of-presheaves}.  Thus
$\mc{F}\subseteq \ol{\mc{F}}$. 
\begin{rem}
  Since $\Psh^{\mc{F}}(C)$ contains every representable presheaf, any
  terminal object of it is also terminal in $\Psh(C)$.  So we could
  instead say that $\ol{\mc{F}}$ 
  consists of $C$ such that its free $\mc{F}$-colimit completion
  $\Psh^{\mc{F}}(C)$ has a terminal object. 
\end{rem}

The significance of regular closure can be illustrated with an
elementary example.
\begin{exam}
  Recall that any coequalizer can be built from pushouts and binary
  coproducts by a simple recipe:
  \[
  \colim(f,g\colon A\rightrightarrows B) = \colim\bigl( A \xla{(\id,\id)}
    A\amalg A \xra{(f,g)} B\bigr).
\]
The point is that there is a \emph{universal example} of this recipe. 
  Consider the finite categories 
  $C=\{\bullet\rightrightarrows \bullet\}$, $P=\{\bullet\leftarrow
  \bullet\rightarrow \bullet\}$, and 
  $Q=\{\bullet,\bullet\}$ (the walking parallel pair of arrows, the
  walking span, 
  and the discrete category with two objects).
Then, applying the above recipe to the universal coequalizer (i.e.,
the colimit of $\rho\colon C\ra \Psh(C)$), we  compute that
\[
1\approx\colim_C\rho \in \Psh^{\mc{F}}(C) \qquad \text{where} \qquad
\mc{F}=\{P,Q\}. 
\]
Therefore we have that $C\in \ol{\mc{F}}$, which concisely encodes the
observation that any coequalizer can be built from colimits with
shapes in $\mc{F}$.
\end{exam}

The notion of regular closure gives a complete answer to the following
question: 
if a full subcategory of an  $\infty$-category is stable under
$\mc{F}$-colimits, what other kinds of 
colimits is it necessarily stable under?

\begin{prop}\label{prop:stable-colims-regular-closure}
  Let $\mc{F}\subseteq \icat$ be a class of small $\infty$-categories,
  and let $A\subseteq B$ be a full subcategory of an $\infty$-category
  $B$ which has all $\mc{F}$-colimits.  If
  $A$ is stable under $\mc{F}$-colimits in $B$, then it is also 
  stable under $\ol{\mc{F}}$-colimits.

  Furthermore, the above statement is the best possible, in the sense
  that if the previous sentence holds with  ``$\ol{\mc{F}}$'' replaced
  with some class $\mc{G}\subseteq \icat$, then we must have $\mc{G}\subseteq
  \ol{\mc{F}}$.  
\end{prop}
\begin{proof}
  First we show that $A$ is stable under $\ol{\mc{F}}$-colimits.
  Using the embedding theorem \eqref{thm:colimit-embedding}, we can
  choose a fully faithful $i\colon B\rightarrowtail B'$ such that $B'$
  is cocomplete and $i$ preserves all $\mc{F}$-colimits.  
Given 
  any functor $f\colon J\ra A$ with $J\in \ol{\mc{F}}$, we 
  want to show that the colimit of $f$ in $B'$ is actually in $A$.  
By
  the universal property $f$ extends over $\rho\colon J\rightarrowtail
  \Psh(J)$ to a colimit preserving functor $\wh{f}\colon \Psh(J)\ra
  B'$, so that the 
  colimit of  $f$ in $B'$ is
  equivalent to $\wh{f}(1)$ \eqref{cor:universal-colim}. Furthermore,
  $\wh{f}(\PSh^{\mc{F}}(J))\subseteq A$ since $A$ is stable under
  $\mc{F}$-colimits.  The claim follows since $J\in \ol{\mc{F}}$ so $1\in
  \Psh^{\mc{F}}(J)$.

Now suppose we know that stability under $\mc{F}$-colimits implies
stability under $\mc{G}$-colimits.  
  Suppose $J\in \mc{G}$ and consider $A\defeq \Psh^{\mc{F}}(J)$ and $B\defeq
  \Psh(J)$.    By hypothesis $A$ is stable under $\mc{G}$-colimits,
  and thus in particular $1\approx \colim_J \rho \in 
  \Psh^{\mc{F}}(J)$, so $J\in \ol{\mc{F}}$ as desired.
\end{proof}

Here is a variant characterization of regular closure, answering the
question: if an 
$\infty$-category has (or a functor preserves) $\mc{F}$-colimits, what
other kinds of colimits must it have (or preserve).
\begin{prop}\label{prop:have-and-preserve-colims-regular}
  Let $\mc{F}\subseteq \icat$ be a class of small $\infty$-categories.
  \begin{enumerate}
  \item Any $\infty$-category $A$ which has $\mc{F}$-colimits also has
    $\ol{\mc{F}}$-colimits.
    \item Let  $f\colon A\ra B$ be a functor between categories which
      have $\mc{F}$-colimits.  If $f$ preserves $\mc{F}$-colimits,
      then $f$ also preserves $\ol{\mc{F}}$-colimits.
    \end{enumerate}
   Furthermore the above is the best possible, in the sense that if
   (1) and (2) hold with  ``$\ol{\mc{F}}$'' replaced by some class
   $\mc{G}\subseteq \icat$, then $\mc{G}\subseteq\ol{\mc{F}}$.
\end{prop}
\begin{proof}
  To prove (1), choose any $\ol{\mc{F}}$-colimit preserving embedding
  $A\rightarrowtail 
  B$ to a cocomplete $\infty$-category \eqref{thm:colimit-embedding} and apply
  \eqref{prop:stable-colims-regular-closure}.     To prove (2), first
  note that by (1) both $A$ and $B$ have $\ol{\mc{F}}$-colimits.  Now
  apply the path-category criterion for colimit preservation (proved in
  the appendix
  \eqref{prop:path-criterion-for-colimit-preservation}),
  which says that $f$ preserves $J$-colimits for some
  $\infty$-category $J$ if and only if a certain 
  fully faithful functor  $\phi\colon \Path(f)\rightarrowtail
  \LPath(f)$ of ``path categories'' is stable under 
  $J$-colimits.  By hypothesis  $j$ is stable under $\mc{F}$-colimits,
  and so is stable under $\ol{\mc{F}}$-colimits by
  \eqref{prop:stable-colims-regular-closure}.

  To see that this is best possible, suppose that having
  $\mc{F}$-colimits implies the existence of $\mc{G}$-colimits, and
  preserving $\mc{F}$-colimits implies the preservation of
  $\mc{G}$-colimits.  Then in particular stability under
  $\mc{F}$-colimits implies stability under $\mc{G}$-colimits, so the
  claim follows from  \eqref{prop:stable-colims-regular-closure}.  
\end{proof}

These ideas give a characterization of regular closure in terms of
free colimit completion.
\begin{prop}\label{prop:stability-of-free-colim-completion}
  Let $\mc{F}\subseteq\icat$ be a class of small $\infty$-categories,
  and let $J\in \icat$.  Then $J\in \ol{\mc{F}}$ if and only if
  $\PSh^{\mc{F}}(C)\subseteq \Psh(C)$ is 
  stable under $J$-colimits for all $C\in \icat$.
\end{prop}
\begin{proof}
  By construction $\Psh^{\mc{F}}(C)\subseteq \Psh(C)$ is stable under
  $\mc{F}$-colimits, and thus is stable under $\ol{\mc{F}}$-colimits by
  \eqref{prop:stable-colims-regular-closure}.  
For the converse, take $C=J$  and recall that $1\approx \colim_J
\rho_J$.
\end{proof}

As a consequence, free $\mc{F}$-colimit completion is the same as
free $\ol{\mc{F}}$-colimit completion.

\begin{prop}\label{prop:reg-closures-corr-to-free-colimit-completions}
  Given $\mc{F},\mc{G}\in \icat$, we have that
  $\ol{\mc{F}}=\ol{\mc{G}}$ if and only if
  $\Psh^{\mc{F}}(C)=\Psh^{\mc{G}}(C)$ for all $C\in \icat$.  In
  particular, $\Psh^{\mc{F}}(C)=\Psh^{\ol{\mc{F}}}(C)$.  
\end{prop}

\section{Regular classes}
\label{sec:regular-class}

Let $\mc{F}\subseteq \icat$ be a class of small $\infty$-categories.
We say that $\mc{F}$ is a  \dfn{regular class} if
$\mc{F}=\ol{\mc{F}}$, 
i.e., if for every small
$\infty$-category $C$, we have that 
$C\in \mc{F}$ whenever
$\Psh^{\mc{F}}(C)$ contains the terminal presheaf.

We can now justify the term ``regular closure''.
\begin{prop}
  Let $\mc{F}\subseteq \icat$ be a class of small
  $\infty$-categories.  Then $\ol{\mc{F}}$ is a regular class, and 
  is in fact the smallest regular class containing $\mc{F}$.
\end{prop}
\begin{proof}
  We have already noted that $\mc{F}\subseteq \ol{\mc{F}}$.
    That  $\ol{\mc{F}}$ is a regular class is immediate from the fact
    that $\Psh^{\mc{F}}(C)=\Psh^{\ol{\mc{F}}}(C)$
    \eqref{prop:reg-closures-corr-to-free-colimit-completions}.  
  If $\mc{G}$ is any regular class which contains $\mc{F}$, then for any $J\in \ol{\mc{F}}$ we have
  \[
  1\in \Psh^{\ol{\mc{F}}}(C) =\Psh^{\mc{F}}(C) \subseteq \Psh^{\mc{G}}(C),
\]
and thus $\ol{\mc{F}}\subseteq \mc{G}$.
\end{proof}

This is a convenient place to note that regular classes are closed
under finite products.
\begin{prop}
  Let $\mc{K}$ be a regular class of small $\infty$-categories.  Then
  $1\in \mc{K}$ and $J,K\in \mc{K}$ implies $J\times K\in \mc{K}$.
\end{prop}
\begin{proof}
  That $1\in \mc{K}$ is clear, since the terminal object of
  $\Psh^{\mc{K}}(1)$ is a representable presheaf.
  For closure under pairwise products, note  that $J\times K$-colimits
  can be 
  computed as the composite
  \[
  \Fun(J\times K,\PSh(C))=\Fun(J,\Fun(K,\Psh(C))) \xra{\colim_J}
  \Fun(K,\PSh(C)) \xra{\colim_K} \PSh(C).
\]
Since colimits in functor categories are computed objectwise, if 
$J,K\in \mc{F}$ then $\Psh^{\mc{F}}(C)$ is stable under $J\times
K$-colimits, whence the claim follows from
\eqref{prop:stability-of-free-colim-completion}.
\end{proof}

\begin{rem}
An earlier preprint version of this paper used the terms
\emph{filtering class} and \emph{filtering closure} for what we are here regular
  class and regular closure.  I've come to feel that ``filtering'' terminology should be
reserved 
for concepts closer to the classical notion of filtered categories,
so here I reserve it for the \emph{filtration classes} described in
\eqref{subsec:filtration-and-cofiltration}.  The use of the word
``regular'' is suggested by an analogy with regular cardinals.  In
particular, the regular classes generated by classes of sets of
bounded cardinality correspond exactly to regular cardinals
\eqref{prop:reg-and-irreg-cardinals}. 
\end{rem}

\section{Explicit description of free colimit completion}
\label{sec:explicit-description-free-colimit}

Using the idea of regular closure, we can give an explicit description
of $\PSh^{\mc{F}}(C)$ as a full subcategory of presheaves.
\begin{prop}\label{prop:regular-criterion-for-colim-closure}
  Let $\mc{F}\subseteq\icat$ be a class of small $\infty$-categories,
  and let $C$ be a small $\infty$-category.  Then 
  \[
  \Psh^{\mc{F}}(C)=\set{X\in \Psh(C)}{ C/X\in \ol{\mc{F}}}.
  \]
\end{prop}
\begin{proof}
  By definition of regular closure, we need to show that $X\in
  \Psh^{\mc{F}}(C)$ if and only if $1\in \Psh^{\mc{F}}(C/X)$.
  We make use  of the functor 
  \[
  \wh{\pi_X}\colon \Psh(C/X)\ra \Psh(C),
\]
which corresponds under the equivalence $\kappa\colon \Psh(C)_{/X}\approx
  \Psh(C/X)$ to the evident forgetful functor for the slice
  \eqref{subsec:slices-of-presheaves}.  Because 
  it is equivalent to  such a forgetful functor, it has the following
  properties. 
\begin{enumerate}
\item The functor $\wh{\pi_X}$ both preserves and reflects colimits: a
  small   diagram in $\PSh(C/X)$ is a colimit if and only if its image
  under $\wh{\pi_X}$
  in $\Psh(C)$   is a colimit.
  \item The functor $\wh{\pi_X}$ both preserves and reflects
    representability: a presheaf $F$ on $C/X$ is representable if and
    only if its image $\wh{\pi_X}(F)$ is a representable presheaf on
    $C$.
\item The image of the terminal presheaf under $\wh{\pi_X}$ is
  isomorphic to $X$.
\end{enumerate}
A straightforward consequence of (1) and (2) is that
$\Psh^{\mc{F}}(C/X)$ is precisely equal to the preimage of
$\PSh^{\mc{F}}(C)$ under $\wh{\pi_X}$.  The claim then follows from (3).
\end{proof}

Thus, the formation of free colimit completion is compatible with
taking slices, in the following sense.  
\begin{cor}\label{cor:colim-closure-slice}
  For any class $\mc{F}\in \icat$, any $C\in \icat$, and any presheaf
  $X\in \Psh(C)$, the equivalence $\kappa\colon \Psh(C)_{/X} \ra
  \Psh(C/X)$ restricts to an equivalence 
  $\PSh^{\mc{F}}(C)\times_{\Psh(C)} \Psh(C)_{/X}\ra
  \Psh^{\mc{F}}(C/X)$ of full subcategories.
\end{cor}

\section{Cofinality}
\label{sec:cofinality}

We recall the notion of a \dfn{cofinal functor}\footnote{I'll use term
  as in \cite{lurie-higher-topos}.  Some sources prefer \emph{final
    functor}.  In \cite{kerodon}*{02MZ}, these are
  called  \emph{right cofinal functors}.} $f\colon C\ra D$, as
defined in \cite{lurie-higher-topos}*{4.1.1}  to which we refer for a
definition.  We will only need the following equivalent
characterizations.

\begin{lemma}\label{lemma:cofinal-colim-pres}
  Let $f\colon C\ra D$ be a functor between small $\infty$-categories.  The
  following are equivalent.
  \begin{enumerate}
  \item $f$ is cofinal.
 \item $f^*\colon \Fun(D^\rhd,A)\ra\Fun(C^\rhd, A)$
(restriction along $f$) preserves all colimit cones which exist, for
any $A$.
  \item for every object $d$ in $D$, the pullback $C\times_D
    D_{d/}$ is weakly contractible.
  \end{enumerate}
\end{lemma}
\begin{proof}
  The equivalence (1) $\Leftrightarrow$ (2) is
  \cite{lurie-higher-topos}*{4.1.1.8}, while (1) $\Leftrightarrow$ (3)
  is \cite{lurie-higher-topos}*{4.1.3.1}.
  \end{proof}

We can restate this criterion from the point of view of colimit
completion.
\begin{lemma}\label{lemma:cofinal-via-univ-colimit}
  A functor $f\colon C\ra D$ between small $\infty$-categories is
  cofinal if and only if the colimit of the composite of $C\xra{f}
  D\xra{\rho_D} \Psh(D)$ is the terminal presheaf.
\end{lemma}
\begin{proof}
  That a cofinal functor has this property is immediate from (1)
  $\Rightarrow$ (2) of 
  \eqref{lemma:cofinal-colim-pres} and  the fact
  that $\colim_D \rho_D\approx 1$.
  
  Conversely, suppose $\colim_D \rho_Df\approx 1$.  Since colimits are
  computed pointwise in $\Psh(D)$, we have that the colimit of the
  composite of $C\xra{f} D \xra{\Map_D(d,-)} \igpd$ is contractible
  for every object $d$ of $D$.  This composite is classified by the
  left fibration $C\times_D D_{d/}\ra C$, and its colimit is the weak
  homotopy type of $C\times_D D_{d/}$
  \cite{lurie-higher-topos}*{3.3.4.6}.  Therefore $f$ is cofinal using
  (3) $\Rightarrow$ (1) of 
  \eqref{lemma:cofinal-colim-pres}.
\end{proof}

\begin{lemma}\label{lemma:reg-class-closed-under-cofinality}
  Let $\mc{F}\subseteq \icat$ be a class of small
  $\infty$-categories.  If $f\colon C\ra D$ is a cofinal functor, then
  $C\in \mc{F}$ implies $D\in \ol{\mc{F}}$.
\end{lemma}
\begin{proof}
  Immediate from \eqref{lemma:cofinal-via-univ-colimit}.
\end{proof}

Say that a class $\mc{F}\subseteq \icat$ \dfn{cofinally generates} a
regular class $\mc{K}$ if for every $D\in \mc{K}$ there exists a
cofinal functor $f\colon C\ra D$ with $C\in \mc{F}$.  Note that this
implies $\ol{\mc{F}}=\mc{K}$.

In general, the free $\mc{F}$-colimit completion of a category $C$ can
be produced by an iterative procedure (as illustrated in the proof of
\eqref{prop:small-colimit-completion-is-small}), in which we build up
a subcategory $\Psh(C)$ by starting with representable presheaves, and
successively adjoining colimits of all $\mc{F}$-shaped 
diagrams in the subcategory.
When a class $\mc{F}$ cofinally generates a
regular class, this means exactly that every object in the free
$\mc{F}$-colimit completion can be realized \emph{immediately} as an
$\mc{F}$-colimit diagram of representable presheaves.
\begin{prop}
  Let $\mc{F}\subseteq\icat$ be  class of small $\infty$-categories,
  and $\mc{K}$ a regular class.  The following are equivalent.
  \begin{enumerate}
  \item The class $\mc{F}$ cofinally generates $\mc{K}$.
  \item For any $C\in \icat$, a presheaf $X\in \Psh(C)$ is in
    $\Psh^{\mc{K}}(C)$ if and only if $X$ is a colimit
    in $\PSh(C)$ of a functor of the form  $\rho_C\circ f$, where
    $f\colon J\ra C$ is a
    functor with $J\in \mc{F}$.
  \end{enumerate}
\end{prop}
\begin{proof}
  Suppose (1) holds.  If $X\in
  \Psh^{\mc{F}}(C)$ then $C/X\in \ol{\mc{F}}$
  \eqref{prop:regular-criterion-for-colim-closure}, and since $\mc{F}$
  cofinally generates $\mc{K}$ we may choose a cofinal $u\colon J\ra
  C/X$ with $J\in \mc{F}$.  Then $X\approx \colim_{C/X} \rho_C\pi_X
  \approx \colim_J \rho_C\pi_X u$, thus expressing $X$ as an
  $\mc{F}$-colimit of representables as desired.

  Conversely, suppose (2) holds.  Given any $C\in \mc{K}$ we can apply
  (2) to the terminal object $X=1$ in $\Psh^{\mc{F}}(C)$, obtaining a
  functor $f\colon J\ra C$ so that $J\in \mc{F}$ and $\colim
  \rho_C\circ f\approx 1$ in $\Psh(C)$.  From
  \eqref{lemma:cofinal-via-univ-colimit} we see that this $f$ is
  cofinal, as desired.
\end{proof}

\begin{rem}
  It is \emph{not} the case that all regular  closures are via cofinal
  generation.  Simple counterexamples include $\mc{F}=\varnothing$
  \eqref{subsec:minimal-regular-class} 
and $\mc{F}=\{\Delta^0\amalg
  \Delta^0\}$ \eqref{subsec:binary-coproducts-regular-class}, which
  do not cofinally generate their regular closures.
\end{rem}

\section{Regular classes and $\infty$-groupoids}
\label{sec:reg-class-infty-gpd}

Given an $\infty$-category $C$, we write $\eta\colon C\ra \len{C}$ for a
tautological map to its groupoid completion.
\begin{lemma}\label{lemma:gpd-completion-cofinal}
The tautological map $\eta\colon C\ra \len{C}$ from an
$\infty$-category to its group completion is cofinal.  
\end{lemma}
\begin{proof}
  It suffices to prove this for a particular model of $\eta$.  For
  instance, there exists a factorization $C\xra{j}  C'\xra{p}
  \Delta^0$ into a right anodyne map $j$ followed by a right fibration
  $p$.  Since the target of $p$ is the terminal object, it is actually
  a Kan fibration, so $C'$ is an $\infty$-groupoid, while $j$ is a
  weak equivalence of simplicial sets.  Thus $j$ is a 
  groupoid completion of $C$, and it is a cofinal map since all right
  anodyne maps are cofinal \cite{lurie-higher-topos}*{4.1.1.3}.
\end{proof}

Recall that $\infty$-groupoids are the colimit completion of the
terminal category: $\Psh(1)\approx \igpd$.
Given a class
$\mc{F}\subseteq \icat$, we write $\igpd^{\mc{F}}\defeq
\Psh^{\msc{F}}(1)\subseteq \igpd$.

\begin{prop}\label{prop:regular-class-infty-gpd}
  Let $\mc{F}\subseteq \icat$ be a class of small
  $\infty$-categories.  Then
  \[
  \ol{\mc{F}}\cap \igpd = \len{\ol{\mc{F}}}= \igpd^{\mc{F}},
  \]
  i.e., the class of $\infty$-groupoids in $\ol{\mc{F}}$ is the class
  of groupoid completions of objects of $\ol{\mc{F}}$, which are the
  objects of the full subcategory of $\infty$-groupoids generated
  under $\mc{F}$-colimits by the point.
\end{prop}
\begin{proof}
  Clearly $\ol{\mc{F}}\cap \igpd \subseteq \len{\ol{\mc{F}}}$.
  The groupoid completion map $\eta\colon J\ra \len{J}$ is cofinal
  \eqref{lemma:gpd-completion-cofinal}, so   
  $J\in \ol{\mc{F}}$ implies $\len{J}\in \ol{\mc{F}}$
  \eqref{lemma:reg-class-closed-under-cofinality}, whence 
  $\len{\ol{\mc{F}}}\subseteq\ol{\mc{F}}\cap \igpd$.  We know that
  $X\in \Psh^{\mc{F}}(1)$ if and only if $X\approx X/1 \in
  \ol{\mc{F}}$ \eqref{prop:regular-criterion-for-colim-closure}, so
  $\igpd^{\msc{F}}=\ol{\mc{F}}\cap \igpd$ 
\end{proof}

\section{Constructing regular classes}
\label{sec:cut-out-reg-classes}

\begin{lemma}\label{lemma:reg-class-cut-out-by-stable}
  Let $\{A_i\subseteq B_i\}$ be a collection of full subcategories, in
  which each $B_i$ is a  cocomplete
  $\infty$-category, and let
  \[
  \mc{F}\defeq \set{C\in \icat}{\text{$A_i\subseteq B_i$ is stable under
      $C$-colimits for all $i$}}.  
\]
Then $\mc{F}$ is a regular class.
\end{lemma}
\begin{proof}
  We need to show that $1\in \Psh^{\mc{F}}(C)$ implies that $C\in
  \mc{F}$.  Given any functor $g\colon C\ra A_i\subseteq B_i$, consider
  the $\mc{F}$-colimit preserving extension $\wh{g}\colon
  \Psh^{\mc{F}}(C)\ra B_i$ of $g$.  By hypothesis the image of $\wh{g}$
  is contained in $A_i$, and since $1\in \Psh^{\mc{F}}(C)$ we have
  $\colim_C g \approx \wh{g}(1)\in A_i$.  Thus, we have shown that
  every  $A_i$ is
  stable under $C$  colimits in $B_i$, so $C\in \mc{F}$ as desired. 
\end{proof}

\begin{prop}\label{prop:reg-class-cut-out}
  Let $\{f_i\colon A_i\ra B_i\}$ be a collection of functors between
  $\infty$-categories, where each $A_i$ and $B_i$ is cocomplete.  Define
  \[
  \mc{F}\defeq \set{C\in \icat}{\text{$f_i$ preserves $C$-colimits
      for all $i$}}. 
\]
Then $\mc{F}$ is a regular class.
\end{prop}
\begin{proof}
We again  use the path-category criterion for colimit preservation
\eqref{prop:path-criterion-for-colimit-preservation}, which implies that $\mc{F}$ is  precisely
the class of $C$ such 
that each of the full subcategories $\Path(f_i)\subseteq \LPath(f_i)$ is stable
under $C$-colimits.  The claim is then immediate from  
\eqref{lemma:reg-class-cut-out-by-stable}.
\end{proof}

In the situation of the previous proposition, we will say that the
regular class $\mc{F}$ is \dfn{cut out} by the collection of
embeddings $\{A_i'\subseteq A_i\}$.  It is easy to see that every
regular class $\mc{F}$ arises in this way: it is cut out by
$\{\Psh^{\mc{F}}(C) \subseteq \Psh(C)\}_{C\in \icat}$, since if
$\Psh^{\mc{F}}(C)\subseteq \Psh(C)$ is stable under $C$-colimits then
$1\in \Psh^{\mc{F}}(C)$.   We have the following immediate consequence.

\begin{cor}\label{cor:intersection-reg-classes}
The intersection of any collection of regular classes is a regular class.
\end{cor}

\subsection{Filtration and cofiltration classes}
\label{subsec:filtration-and-cofiltration}

Examples of the above construction are the \emph{filtration} and
\emph{cofiltration} classes.

Given a class $\mc{U}\subseteq \icat$ of small $\infty$-categories, we
define its associated \dfn{filtration class} to be 
\[
\Filt(\mc{U})=\set{J\in \icat}{\text{$\llim_{U^\op}\colon \Fun(U^\op,
    \igpd)\ra \igpd$ preserves $J$-colimits for all $U\in \mc{U}$}}. 
\]
Likewise, given a class $\mc{J}\subseteq \icat$ of small
$\infty$-categories, we define its associated \dfn{cofiltration class}
to be
\[
\coFilt(\mc{J})=\set{U\in \icat}{\text{$\colim_{J}\colon
    \Fun(J,\igpd)\ra \igpd$ preserves $U^\op$-limits for all $J\in
    \mc{J}$}}. 
\]
Clearly both filtration classes and cofiltration classes are examples
of regular classes.  Note further that
\[
\mc{U}\subseteq \mc{V} \quad \Longrightarrow \quad
\Filt(\mc{U})\supseteq \Filt(\mc{V}), \quad
\coFilt(\mc{U})\supseteq\coFilt(\mc{V}), 
\]
and
\[
\mc{U}\subseteq \coFilt(\mc{J}) \quad \Longleftrightarrow \quad
\mc{J}\subseteq \Filt(\mc{U}),
\]
that is, $\Filt$ and $\coFilt$ define a ``Galois connection'' on the
collection of regular classes.

Filtration classes include the classes of $\kappa$-filtered
categories, sifted categories, and others: see examples
\eqref{subsec:idempotent-completion}--\eqref{subsec:weakly-contractible-cats} below.

\section{A recognition principle for free colimit completion}
\label{sec:recognition-principle}

Let $\mc{F}\subseteq \icat$ be a class of small $\infty$-categories,
and $A$ an $\infty$-category which has $\mc{F}$-colimits.  Say that an
object $a$ of $A$ is \dfn{$\mc{F}$-compact} if
\[
\Map_A(a,-)\colon A\ra \igpd
\]
preserves all $\mc{F}$-colimits.   I write
$A^{\mc{F}\mr{-cpt}}\subseteq A$ for the full subcategory of
$\mc{F}$-compact objects.

The notion of $\mc{F}$-compactness really only depends on the
regular closure of $\mc{F}$.
\begin{prop}\label{prop:f-compact-regular}
  Let $A$ be an $\infty$-category which has $\mc{F}$-colimits.  Then
  an object of $A$ is $\mc{F}$-compact if and only if it is
  $\ol{\mc{F}}$-compact.  
\end{prop}
\begin{proof}
  Apply   \eqref{prop:have-and-preserve-colims-regular} to $A$ and to 
  $\Map_A(a,-)$.   
\end{proof}
Thus, if $a$ is $\mc{F}$-compact, then $\Map_A(a,-)$ preserves
$\ol{\mc{F}}$-colimits.

We obtain the following recognition principle.\footnote{This general
  principle was already known to Jacob Lurie (personal communication).}
\begin{prop}\label{prop:recognition-free-colimit-completion}
  Let $\mc{F}\subseteq \icat$ be a class of small
  $\infty$-categories, and suppose $C\in\icat$. 
  Let $\wh{f}\colon \PSh^{\mc{F}}(C)\ra A$ be an $\mc{F}$-colimit
  preserving  functor to an  $\infty$-category which has
  $\mc{F}$-colimits, and let $f=\wh{f}\rho_C\colon C\ra A$.
  \begin{enumerate}
  \item If $f$ is fully-faithful and $f(C)\subseteq
    A^{\mc{F}\mr{-cpt}}$, then $\wh{f}$ is fully
    faithful.
  \item
    The functor $\wh{f}$ is an equivalence if and only if
    \begin{enumerate}
    \item [(i)]$f$ is fully faithful.
    \item [(ii)] $f(C)\subseteq A^{\mc{F}\mr{-cpt}}$.
      \item [(iii)] The objects of $f(C)$ generate $A$ under $\mc{F}$-colimits.
      \end{enumerate}
  \end{enumerate}
\end{prop}
\begin{proof}
  Without loss of generality we can replace  $\mc{F}$ with
  $\ol{\mc{F}}$, using 
\eqref{prop:have-and-preserve-colims-regular} and
  \eqref{prop:f-compact-regular}.  Furthermore, we 
  know that $\Psh^{\mc{F}}(C)=\set{X\in
    \Psh(C)}{C/X\in \ol{\mc{F}}}$
  \eqref{prop:regular-criterion-for-colim-closure}, and that every $X$ is
  tautologically a $C/X$-colimit of representable presheaves.
  
  Then this is proved exactly as in \cite{lurie-higher-topos}*{5.3.5.11},
  which deals with the special case where $\mc{F}$ is the class of
  $\kappa$-filtered $\infty$-categories for some regular cardinal $\kappa$.
\end{proof}

\begin{rem}
  Note that the original example $C\xra{\rho} \Psh^{\mc{F}}(C)\subseteq \Psh(C)$
  of a free $\mc{F}$-colimit completion is exactly of this type, since
  all representable presheaves are ``completely compact''
  \cite{lurie-higher-topos}*{5.1.6.2}.   In the case of $\mc{F}=\icat$
  this  recovers \cite{lurie-higher-topos}*{5.1.6.11}.
\end{rem}

\section{Examples of regular classes and regular closures}
\label{sec:examples}

\subsection{The minimal regular class}
\label{subsec:minimal-regular-class}

Since $\Psh^{\varnothing}(C)\approx C$, we have that
\[
\ol{\varnothing} = \set{C\in \icat}{\text{$C$ has a terminal object}}.
\]
This is a (very) trivial example of a regular closure which is not of cofinal
generation (\S\ref{sec:cofinality}).  

\subsection{The maximal regular class}

Clearly $\icat$ is a regular class, and $\Psh^{\icat}(C)=\Psh(C)$.
The $\icat$-compact objects of a cocomplete $\infty$-category are
precisely what are called \emph{completely compact} in
\cite{lurie-higher-topos}*{5.1.6.2}.

There are a number of well-known identifications of $\icat$ as a
regular closure.  For instance, it is the regular closure of  $\Set
\cup \{ \Lambda^2_1\}$ (i.e., coproducts and pushouts, see
\cite{lurie-higher-topos}*{4.4.2.6}).  It is the cofinal closure of
the class of small 1-categories \cite{lurie-higher-topos}*{4.2.3.14}, and in
fact the cofinal closure of the class of small posets
\cite{lurie-higher-topos}*{4.2.3.15}.

\subsection{Coproducts}

Let $\Set\subseteq \icat$ be the collection of all small and discrete
$\infty$-groupoids, so that $\Psh^\Set(C)$ is the free completion of
$C$ with respect to small coproducts.  It is straightforward to show
that $\Psh^\Set(C)$ consists exactly of presheaves which are
equivalent to small coproducts of representables, as this subcategory is
itself clearly stable under coproducts.  Thus $1\in \Psh^\Set(C)$
implies $1\approx \coprod_i \rho(c)$, and using this you can show that 
\[
\ol{\Set} = \set{\coprod_i C_i}{\text{each $C_i\in \icat$ has a
    terminal object}}.  
\]
In particular, $\ol\Set$ is cofinally generated by $\Set$.

Analogous considerations identify the regular
closure of $\kappa$-small sets, where $\kappa$ is any infinite 
cardinal, though the precise description depends on whether $\kappa$ is a regular cardinal.
In the following, I write $\kappa^+$ for the successor cardinal of
$\kappa$, and $\Set^{<\kappa}\subseteq\Set$ for the class of  small
and discrete $\infty$-groupoids with fewer than $\kappa$ path-components.

\begin{prop}\label{prop:reg-and-irreg-cardinals}
  Let $\kappa$ be an infinite cardinal.
  \begin{enumerate}
  \item If $\kappa$ is regular, then $\ol{\Set^{<\kappa}}$ is cofinally
      generated by $\Set^{<\kappa}$, and is not equal to 
      $\ol{\Set^{<\kappa^+}}$.
    \item If $\kappa$ is irregular, then
      $\ol{\Set^{<\kappa}}=\ol{\Set^{<\kappa^+}}$, and is not
      cofinally generated by $\Set^{<\kappa}$.
  \end{enumerate}
\end{prop}
\begin{proof}
  From \eqref{prop:regular-class-infty-gpd} we have that both 
  $\ol{\Set^{<\kappa}}\cap \igpd$ and $\len{\ol{\Set^{<\kappa}}}$ are
  equal to $\igpd^{\Set^{<\kappa}}$, the full
  subcategory of $\infty$-groupoids generated by the point under
  $\Set^{<\kappa}$-colimits.   If $\kappa$ is regular, then this
  subcategory is exactly $\Set^{<\kappa}$.  If
  $\kappa$ is irregular,  this subcategory contains a set of cardinality
  not bounded by $\kappa$.  Since successor cardinals are always
  regular, $\igpd^{\Set^{<\kappa}}$ must in this case be 
  $\Set^{<\kappa^+}$.  The claims of the proposition are now
  straightforward using the above remarks on $\ol{\Set}$.
\end{proof}

\subsection{Binary coproducts}
\label{subsec:binary-coproducts-regular-class}

Let $\mc{F}=\{\Delta^0\amalg \Delta^0\}$, so that $\PSh^{\mc{F}}(C)$ is
the free completion of $C$ with respect to pairwise coproducts.  Then
\[
\ol{\mc{F}}= \set{\coprod_{i\in I} C_i}{\text{$I$ is finite and
    non-empty, and each $C_i\in \icat$ has a terminal object}}. 
\]
Note that $\mc{F}$ does not cofinally generate $\ol{\mc{F}}$.

\subsection{Idempotent completion}
\label{subsec:idempotent-completion}

Let $\Idem$ be the walking idempotent.  Then $\Psh^{\{\Idem\}}(C)$ is
an idempotent completion of $C$ \cite{lurie-higher-topos}*{5.3.6.9}.
Thus
\[
\ol{\{\Idem\}} = \set{C\in \icat}{\text{the idempotent completion of
    $C$ has a terminal object}}. 
\]
Then $\ol{\{\Idem\}}=\Filt(\icat)$, i.e., 
$J$-colimits of $\infty$-groupoids preserve all small limits if and
only if $\Psh^{\{\Idem\}}(J)$ has a terminal object. 
(It is easy to see that $\ol{\{\Idem\}}\subseteq \Filt(\icat)$.  For
the converse, note that $\colim_J\colon\Fun(J,\igpd)\ra \igpd$ is
accessible, so is corepresented by some $A\in \Fun(J,\igpd)$
\cite{lurie-higher-topos}*{5.5.2.7}.  That $J\in \Filt(\icat)$ means
that $A$ is completely compact and so is a retract of some
corepresentable functor $\Map_J(j,-)\colon J\ra \igpd$
\cite{lurie-higher-topos}*{5.1.6.8}. 
Finally, $\Map_{\Fun(J,\igpd)}(A,\Map_J(j,-))\approx \colim_J
\Map_J(j,-)\approx *$, whence $A$ corresponds to an initial object of
the idempotent completion of $J^\op$.)

\subsection{$\kappa$-filtered $\infty$-categories}

Given a regular cardinal $\kappa$, let $\Sm_\kappa$ denote the class
of $\kappa$-small $\infty$-categories, i.e., ones which are equivalent
to a $\kappa$-small simplicial set.  Consider the filtering class
$\Filt(\Sm_\kappa)$, which consists of all $J\in \icat$ such that
$J$-colimits of $\infty$-groupoids preserve finite limits.

Then 
$\Filt(\Sm_\kappa)$ is precisely the collection of all small
\emph{$\kappa$-filtered} $\infty$-categories
\cite{lurie-higher-topos}*{5.3.1.7}, i.e., those $J$ such that $K\ra
J$ extends over $K\subseteq K^\rhd$ for every $\kappa$-small
simplicial set $K$.  (See \cite{lurie-higher-topos}*{5.3.3.3} for a proof.)

In particular, \eqref{prop:regular-criterion-for-colim-closure} says
that $\Psh^{\Filt(\Sm_\kappa)}(C)=\Ind_\kappa(C)$, where the latter
is as in \cite{lurie-higher-topos}*{5.3.5}.
Objects are $\Filt(\Sm_\kappa)$ compact if and
only if they are $\kappa$-compact in the usual sense
\cite{lurie-higher-topos}*{5.3.4}.
Thus, our recognition principle
\eqref{prop:recognition-free-colimit-completion} for free 
$\kappa$-filtered-colimit completion recovers Lurie's
\cite{lurie-higher-topos}*{5.3.5.11}.

The class $\Filt(\Sm_\kappa)$ is cofinally generated by the class of
$\kappa$-directed sets \cite{lurie-higher-topos}*{5.3.1.18}. 

\subsection{Sifted $\infty$-categories}

Let $\Set^{<\omega}\subseteq \icat$ denote the class of finite
discrete $\infty$-groupoids.  Consider the filtering class
$\Filt(\Set^{<\omega})$,  which consists of all $J\in \icat$ such
that $J$-colimits of $\infty$-groupoids preserve finite products.

Then $\Filt(\Set^{<\omega})$ 
is precisely the collection of  all small
\emph{sifted} $\infty$-categories \cite{lurie-higher-topos}*{5.5.8.1},
i.e., those $J$ such that (i) $J$ is non-empty, and (ii) the diagonal
$\delta\colon J\ra J\times J$ is cofinal.  (This equivalence is
well-known.  The main part of the proof is
\cite{lurie-higher-topos}*{5.5.8.11-12}.)

Sifted colimit completion of $C$ is studied in
\cite{lurie-higher-topos}*{5.5.8},  in the special case when $C$
itself is assumed to have finite coproducts.

\subsection{Distilled $\infty$-categories}
\label{subsec:distilled}

Consider the filtering class $\Filt(\{\Lambda^2_0\})$, which consists
of all $J\in \icat$ such that $J$-colimits of $\infty$-groupoids
preserve pullbacks.  This is exactly the class of small
\emph{distilled} $\infty$-categories, where we say that $J$
is \dfn{distilled} if for every functor $f\colon 
\Lambda^2_0\ra J$, the slice $J_{f/}$ has contractible weak homotopy
type.  This identification is proved in
\cite{rezk-gen-accessible-infty-cats}. 
Results of that paper show that the class of small distilled
$\infty$-categories is the regular closure of
$\Filt(\Sm_\omega)\cup \igpd$.

\subsection{Weakly contractible $\infty$-categories}
\label{subsec:weakly-contractible-cats}

The filtering class $\Filt(\{\varnothing\})$ consists of $J$ such
that $J$-colimits of $\infty$-groupoids preserve the terminal object,
i.e., such that $ \colim_J *$ is contractible.  These are precisely the weakly
contractible $\infty$-categories.

\subsection{Other examples?}

The problem of determining $\ol{\mc{F}}$ when $\mc{F}$
is not already known to be a regular class is 
unexplored.  In practice, one would often like to determine if
$\ol{\mc{F}}$ is the cofinal closure of $\mc{F}$.  
Here are some 
interesting possibilities to consider.
\begin{itemize}

  \item $\mc{F}=\igpd$, the class of $\infty$-groupoids.  We will
    prove in subsequent work that $\ol{\igpd}$ consists exactly of
    those $C$ such that the map 
    $u\colon C\ra \len{C}$ to its groupoid completion is a left
    adjoint (or what is the same thing: such that there exists a
    cofinal functor $G\ra C$ from an $\infty$-groupoid).

  \item $\mc{F}=\Sm_\omega$, the class of $\omega$-small
  $\infty$-categories, so that $C\ra \Psh^{\Sm_\omega}(C)$ is a free
  finite-colimit completion of $C$.   It seems plausible that
  $\Sm_\omega$ cofinally generates its regular closure.
\end{itemize}

\section{Appendix: Functors preserving colimits}
\label{sec:functors-preserving-colimits}

In this section I prove a criterion for preservation of colimits by
functors which is surely well-known, but for which I have no
convenient reference.  This proof was suggested to me by Maxime Ramzi.

Given a functor $f\colon A\ra B$ of $\infty$-categories, define the \dfn{oplax
  path category} of $f$ to be
\[
\LPath(f) \defeq (B\times A)\times_{B\times B} \Fun(\Delta^1,B),
\]
so that objects of $\LPath(f)$ correspond to triples
$(b,a,\gamma\colon b\ra f(a))$ 
with $a$ an object of $A$, $b$ an object of $B$, and $\gamma$ a
morphism of $B$.  That is, $\LPath(f)$ is the ``comma category'' of
the pair of functors $B\xra{\id} B \xla{f} A$.  

Write $\Path(f)\subseteq \LPath(f)$ for the
\dfn{path category}, i.e., the full
subcategory spanned by $(b,a,\gamma)$ such that $\gamma$ is an
isomorphism in $B$.  
We write $\pi_A\colon \LPath(f)\ra A$ and $\pi_B\colon
\LPath(f)\ra B$ for the evident projection functors.  Note that the
restriction $\Path(f)\ra A$ of $\pi_A$ is an equivalence.

\begin{prop}\label{prop:path-criterion-for-colimit-preservation}
  Let $J$ be an $\infty$-category, and suppose $f\colon A\ra B$ is a
  functor between $\infty$-categories 
  which have $J$-colimits.  The following are equivalent.
  \begin{enumerate}
  \item The functor $f$ preserves all $J$-colimits.
    \item The full subcategory $\Path(f)\subseteq \LPath(f)$ is stable
      under $J$-colimits.
  \end{enumerate}
\end{prop}

We will prove this using the following lemma, whose proof we defer to
the end of the section.
\begin{lemma}\label{lemma:path-category-technical}
  Let $K$ be a simplicial set, and let $g\colon K\ra \LPath(f)$ be a
  map.  If $\pi_A g$ and $\pi_B g$ admit colimits in $A$ and $B$
  respectively, then $g$ has a colimit, and both $\pi_A$ and $\pi_B$
  preserve such colimits.
\end{lemma}

\begin{proof}[Proof of Corollary from the Lemma]
  (1) $\Longrightarrow$ (2).  Suppose $g\colon J^\rhd\ra \LPath(f)$ is a
  colimit diagram such that $g(J)\subseteq \Path(f)$.  Using
  \eqref{lemma:path-category-technical} we see that  $g$
  corresponds to a triple $(\beta,\alpha,\gamma)$ where $\alpha\colon
  J^\rhd\ra A$ and 
  $\beta\colon J^\rhd \ra B$ are colimit diagrams, and 
  $\gamma\colon \beta\ra f \alpha$ is a natural transformation of
  functors $J^\rhd\ra B$ such that 
  $\gamma|_J$ is a natural isomorphism.  Since by hypothesis $f$
  preserves $J$-colimits, 
  $f\alpha \colon J^\rhd\ra B$ is also colimit diagram, so $\gamma$
  must be an isomorphism, i.e., $g(J^\rhd)\subseteq \Path(f)$.

  (2) $\Longrightarrow$ (1).  Suppose $\alpha\colon J^\rhd\ra A$ is a
  colimit diagram.  Since $B$ has $J$-colimits, we can construct a map
  $g\colon J^\rhd\ra \LPath(f)$ corresponding to the triple
   $(\beta,\alpha,\gamma)$, where $\beta\colon J^\rhd\ra B$ is a
   colimit diagram, and $\gamma\colon \beta\ra f\alpha$ is such that
   $\gamma|_J$ is an isomorphism of functors $J\ra B$.  By
   \eqref{lemma:path-category-technical} the map $g$ is a colimit
   diagram such that $g(J)\subseteq \Path(f)$, and thus by hypothesis
   $g(J^\rhd)\subseteq \Path(f)$, whence $\gamma$ is an isomorphism of
   functors, so $f\alpha$ is a colimit as desired.
 \end{proof}

Now we turn to the proof of the Lemma.  We note that it is a special case of a more
general principle for lax and oplax limits of $\infty$-categories.
Given a functor $\phi\colon C\ra \icat$, we can form the lax limit
$\laxlim_C \phi$, or the analogous oplax limit
$\oplaxlim_C \phi$, as in \cite{gepner-haugseng-nikolaus-lax}.  For
instance, our oplax path category is the oplax 
limit of $\phi\colon \Delta^1\ra \icat$ which represents the diagram
$(f\colon A\ra B)$.   The general principle asserts that limits in
$\laxlim_C\phi$ can be ``constructed and computed pointwise'', and
similarly for colimits in $\oplaxlim_C \phi$.   The first statement is
proved as \cite{linksens-globalizing-global}*{3.9}, while the second
statement is formally dual to it.  Below I'll give a version of this
argument in the special case that we need. 

\begin{proof}[Proof of the Lemma]
  Fix a map of simplicial sets $f\colon A\ra B$.  Then there is a map
  $p\colon X\ra \Delta^1$ of simplicial sets with the following properties.
  \begin{enumerate}
    \item The simplicial set $\Fun_{\Delta^1}(\Delta^1,X)$ of global
      sections is isomorphic to $\LPath(f)$.
      \item This isomorphism restricts to isomorphisms of simplicial
        sets  $X_0\approx B$ and
    $X_1\approx A$, where  $X_s=p^{-1}(s)$ is the fiber of $p$ over
    $s$.
    \item If $A$ and $B$ are quasicategories, then $p$ is a Cartesian fibration.
    \end{enumerate}
This is by an explicit construction, essentially the \emph{weighted nerve}
construction of \cite{kerodon}*{025W} (but adapted to
produce a Cartesian fibration rather than a coCartesian one).
Alternately, it is an example of  the \emph{quasi-categorical collage}
construction of
\cite{riehl-verity-elements-infty-cat}*{F.5.2}, written
$\mathrm{col}(\id_B,f)$ in their notation.
Explicitly, an $n$-cell in $X$ is data $(\sigma, c_A,c_B)$, consisting of
\begin{itemize}
\item a map $\sigma\colon\Delta^n\ra \Delta^1$ of simplicial sets, and
\item a commutative diagram of simplicial sets
  \[\xymatrix{
      {\sigma^{-1}(1)} \ar@{>->}[r] \ar[d]_{c_A}
      & {\Delta^n} \ar[d]^{c_B}
      \\
      {A} \ar[r]_{f}
      & {B}
    } \]
  where $\sigma^{-1}(1)\subseteq \Delta^n$ is the preimage of the final
  vertex of $\Delta^1$.  
\end{itemize}
The identifications (1) and (2) are elementary, while (3) is
essentially 
\cite{kerodon}*{5.3.3.16} or
\cite{riehl-verity-elements-infty-cat}*{F.5.4}, though both of those
references rather express the complementary formulation appropriate for
coCartesian fibrations.

Given a map $g\colon K\ra
\LPath(f)=\Fun_{\Delta^1}(\Delta^1,X)$, write $g_s\colon K\ra X_s$ for
the restriction the fiber over $s\in \{0,1\}$.   Suppose that both 
$g_0$ and $g_1$ admit colimits in their respective fibers.  To prove
the Lemma, we need to show the following.
\begin{enumerate}
\item There exists an extension $\ol{g}\colon K\diamond \Delta^0\ra
  \Map_{\Delta^1}(\Delta^1,X)$ of $g$ which restricts to a colimit
  diagram $\ol{g}_s \colon K\diamond \Delta^0\ra X_s$ in each fiber.
  (Here $K\diamond L$ is the 
  alternate join of \cite{lurie-higher-topos}*{4.2.1}, which in this
  case is just $K\diamond \Delta^0=K\times\Delta^1/K\times \{1\}$.)

\item An extension $\ol{g}\colon K\diamond\Delta^0\ra
  \Map_{\Delta^1}(\Delta^1,X)$ of $g$ is a colimit of $g$ if and only
  if each $\ol{g}_s\colon K\diamond \Delta^0\ra X_s$ is a colimit for $g_s$.
\end{enumerate}
This is precisely the statement of \cite{lurie-higher-topos}*{5.1.2.2}
in the special case that the target of $p$ is $\Delta^1$.
\end{proof}


\end{document}